\documentclass[12pt]{amsart}

\usepackage{amsfonts, amsmath, amscd}
\usepackage[psamsfonts]{amssymb}

\usepackage{amssymb}

\usepackage{enumerate}

\usepackage{pb-diagram}
\usepackage[usenames]{color}

\usepackage{amssymb}
\usepackage{amsfonts}
\usepackage{graphicx}
\usepackage{amscd}

\makeatletter
\@namedef{subjclassname@2010}{%
  \textup{2010} Mathematics Subject Classification}
\makeatother

\numberwithin{equation}{section}\numberwithin{equation}{section}

\newtheorem{theorem}{Theorem}[section]
\newtheorem{corollary}[theorem]{Corollary}

\newtheorem{proposition}[theorem]{Proposition}

\newtheorem{remark}[theorem]{Remark}
\newtheorem{lemma}[theorem]{Lemma}
\newtheorem{example}[theorem]{Example}

\newtheorem{conjecture}[theorem]{Conjecture}

\newcommand{\w}{\omega}

\newcommand{\R}{{\mathbb R}}

\newcommand{\N}{{\mathbb N}}
\newcommand{\Q}{{\mathbb Q}}

\newcommand{\e}{{\epsilon}}

\newcommand{\abs}[1]{\lvert#1\rvert}

\def\sym#1{{[\hspace{-.28em}(#1)\hspace{-.28em}]}}

\begin{document}

\title[On topological groups admitting a $\omega^{\omega}$-base]{On topological groups admitting a base at identity indexed with $\omega^{\omega}$}
\author[Arkady G. Leiderman, Vladimir G. Pestov, and Artur H. Tomita]{Arkady G. Leiderman, Vladimir G. Pestov, and Artur H. Tomita}
\address{A.G.L.: Department of Mathematics, Ben-Gurion University of the Negev, Beer Sheva, P.O.B. 653, Israel}
\email{arkady@math.bgu.ac.il}
\address{V.G.P.: Department of Mathematics and Statistics, University of Ottawa, 585 King Edward Avenue, Ottawa, Ontario K1N 6N5, Canada;}
\address{Departamento de Matem\'atica, Universidade Federal de Santa Catarina, Trindade, Florian\'opolis, SC, 88.040-900, Brazil}
\email{vpest283@uottawa.ca}
\address{A.H.T: Instituto de Matem\'atica e Estatistica, Universidade de S\~ao Paulo, Rua do Matao, 1010, CEP 05508-090, S\~ao Paulo, Brazil}
\email{tomita@ime.usp.br}
\keywords{Topological group, uniform space, free topological group, Tukey order}
\subjclass[2010]{Primary 22A05, 54H11; Secondary 06A06}

\date{\today}

\begin{abstract}
        A topological group $G$ is said to have a local $\w^\w$-base if the neighbourhood system at identity admits a monotone cofinal map from the directed set $\w^\w$. In particular, every metrizable group is such, but the class of groups with a local $\w^\w$-base is significantly wider.  The aim of this article is to better understand the boundaries of this class, by presenting new examples and counter-examples. Ultraproducts and non-arichimedean ordered fields lead to natural families of non-metrizable groups with a local $\w^\w$-base which nevertheless are Baire topological spaces.
        
 More examples come from such constructions as the free topological group $F(X)$ and the free Abelian topological group $A(X)$ of a Tychonoff (more generally uniform) space $X$, as well as the free product of topological groups.
We show that 1) the free product of countably many separable topological groups with a local $\w^\w$-base admits a local $\w^\w$-base;
2) the group $A(X)$ of a Tychonoff space $X$ admits a local $\w^\w$-base if and only if the finest uniformity of $X$
 has an $\w^\w$-base;
3) the group $F(X)$ of a Tychonoff space $X$ admits a local $\w^\w$-base provided $X$ is separable and the finest uniformity of $X$ has
 an $\w^\w$-base.
          
\end{abstract}

\maketitle
\section{Introduction}\label{Intro}
It is a basic fact that a uniform space $(X,{\mathcal U})$ is metrizable if and only if the uniform structure, $\mathcal U$, admits a countable basis. Equivalently, there is a map $V\colon\omega\to {\mathcal U}$ that is both monotone ($n\leq m\Rightarrow V_n\supseteq V_m$) and cofinal (if $V\in {\mathcal U}$, then for some $n$ one has $V_n\subseteq V$). Natural compatible uniform structures exist, for instance, in topological groups. As an immediate consequence, a topological group is metrizable if and only if it admits a countable basis of neighbourhoods of the identity, or, the same, a basis indexed with $\w$. 

It is therefore reasonable to ask which uniform structures admit bases indexed with other common directed sets, which are perhaps not quite as simple as $\w$ and yet yield some extra information. One of the best known and important such sets is the family $\w^\w$ of all sequences of natural numbers equipped with the pointwise partial order.

Uniform spaces admitting a base indexed with $\w^\w$ were considered first by Cascales and Orihuela \cite{CO}. They have proved that compact subspaces with this property are metrizable (recall that a compact space supports a unique compatible uniformity). As an immediate consequence, if a topological group $G$ admits a base at identity indexed with $\w^\w$, then every compact subset of $G$ is metrizable. The main interest of the paper \cite{CO} was in using this observation as a tool of study of locally convex spaces with certain more general properties.
Locally convex spaces with a local base indexed with $\w^\w$ (which was called a $\mathfrak G$-base) were specifically studied in \cite{FKLS}, as well as in the recent monograph \cite{kak}.

A systematic study of topological groups admitting a base at identity indexed with $\w^\w$ was initiated in the papers \cite{GabKakLei_1}, \cite{GabKakLei_2}, see also  \cite{CFHT}, \cite{GabKak_1}, \cite{GabKak_2}.
We adopt a natural terminology  and say that a topological group $G$ admits a {\em local $\w^\w$-base} if there exists a monotone cofinal map from $\w^\w$ to the neighbourhood basis of $G$ at identity. It is of course the same as to say that there exists a neighbourhood basis at identity indexed with the directed set $\w^\w$.
The powerful machinery of duality is available in the class of locally convex spaces, which is lacking in the case of general topological groups. For this reason, the progress for topological groups has been slower.

It is easy to see that the class of topological groups with a local $\w^\w$-base is closed under passing to subgroups, quotient groups, the two-sided completions. A countable product of topological groups with a local $\w^\w$-base equipped with the box product topology has a local $\w^\w$-base. A somewhat less immediate observation is that this class of topological groups is also closed under countable products with the usual product topology, see \cite{GabKakLei_2}, Proposition 2.7. 

We start our Section 2 by observing that the every regular cardinal Tukey-dominated by $\w^\w$ serves as the character (smallest cardinality of a base at identity) of a topological group with a local $\w^\w$-base. This answers Question 2.6 and Question 2.15 of \cite{GabKakLei_2}. We exploit the following observation: every topological group which has a linearly ordered base at identity of cofinality type dominated by $\w^\w$ has a local $\w^\w$-base. Examples of such groups include linearly ordered groups with the order topology, full linear groups over ordered fields, and Bankston ultraproducts of countable families of metrizable groups.
Such groups possess rather unusual combinations of properties. In particular, some of them provide counter-examples to the following question (Question 4.2 of \cite{GabKakLei_2}, Question 9 of \cite{GabKakLei_1}, Problem 2.7 of \cite{GabKak_1}): is it true that a Baire topological group with a local $\w^\w$-base
 (A topological space $X$ is called Baire if the union of every countable collection of closed sets with empty interior in $X$ has empty interior) is necessarily metrizable? In a particular case of topological vector spaces the answer is affirmative \cite{GabKakLei_2}.

In Section 3 we study the following question. Suppose the group topology on a group $G$ is the finest topology making a given family of filters converge to the identity. When does this topology have a local $\w^\w$-base? It is the case provided the group is countable, the family of filters is countable, and each one of them has an $\w^\w$-base. If the group is SIN (the left and right uniformities coincide), then the countability assumption is unnecessary. One  application is to the free products of topological groups as defined by Graev \cite{Gr}: the free product of countably many separable topological groups with a local $\w^\w$-base admits a local $\w^\w$-base (Corollary 3.7). This is new already for the free product $G\ast H$ of two Polish groups. Another application is to the free topological group $F(X)$ and the
free Abelian topological group $A(X)$ of a Tychonoff space $X$ (see \cite{Gra}, \cite{Mar}, \cite{Rai}).  
We observe that $A(X)$ admits a local $\w^\w$-base if and only if the finest uniformity of $X$ has an $\w^\w$-base (Corollary 3.20). The group $F(X)$
 admits a local $\w^\w$-base provided $X$ is separable and the finest uniformity of $X$ has an $\w^\w$-base (Corollary 3.12).   
The above results again lead to new examples of groups with a local $\w^\w$-base, and we manage to partially answer Question 4.15 from \cite{GabKakLei_2}. Some particular cases of those results were obtained previously in \cite{GabKakLei_2} and \cite{GabKak_2}. 

We conclude the article by showing that the free locally convex space $L(X)$ need not have a local $\w^\w$-base even if the free Abelian topological group $A(X)$ does. In fact this happens already for a discrete space $X$ of cardinality continuum, thus we answer negatively Question 4.19 from \cite{GabKakLei_2}.

\begin{remark}
{\em
Gabor Lukacs informed us that he was also interested in the question when the group $A(X)$ admits a local $\w^\w$-base.
He announced some results obtained jointly with Boaz Tsaban (without proofs) in Abstracts of 
IVth Workshop on Coverings, Selections, and Games in Topology, Caserta, Italy (2012). Their work remained unpublished.

Some results of our paper very recently were extended and deepened in several joint preprints by Taras Banakh and Arkady Leiderman;  and in a joint preprint by Saak Gabriyelyan and Jerzy~K{\c{a}}kol.
}
\end{remark}

\section{Order and bases at identity}\label{Ultraproducts}

\subsection{Linear groups over ordered fields}
Topological groups having a linearly ordered basis of neighbourhoods at the identity of a given cofinality type $\tau$ can be obtained in a number of different ways. We start with one such construction, that of linear groups over ordered fields, in order to have an initial supply of examples.

For every regular cardinal $\tau$, there exists a linearly ordered field having $\tau$ as the cofinality type, this is well known and rediscovered a number of times \cite{HM, N}. We will recall just the simplest of possible such constructions.

\begin{example}
{\em
Given a field $k$, denote $k(\alpha)$ a simple transcendental extension of $k$ with a variable $\alpha$. In other words, $k(\alpha)$ consists of all rational functions $p(\alpha)/q(\alpha)$ where $p,q$ are polynomials in $\alpha$ with coefficients in $k$. 
If now $k$ is an ordered field, then $k(\alpha)$ becomes an ordered field with $\alpha$ as a positive infinitesimal, that is, $0<\alpha<x$ for all $x\in k$, $x>0$. The sign of a polynomial $p(\alpha)=a_0+a_1\alpha +\ldots +a_n\alpha^n$ is the sign of the non-zero coefficient of the lowest degree. This extends to the rational functions in an obvious way. Clearly, $k$ is an ordered subfield of $k(\alpha)$.
        
If $\tau$ is a regular cardinal, we construct recursively in $\beta\leq\tau$ an increasing transfinite sequence $k_{\beta}$ of ordered fields, where $k_0=\Q$ (or any other fixed ordered field), $k_{\beta+1}=k_\beta(\alpha_{\beta})$ with $\alpha_{\beta}$ being a positive infinitesimal over $k_{\beta}$, and for limit cardinals $\beta$ we pose $k_{\beta}=\cup_{\gamma<\beta}k_{\gamma}$. It is now easy to see that the ordered field $k_{\tau}$ has the required cofinality $\tau$. 

Viewed as an additive group equipped with the order topology, $k_{\tau}$ becomes an Abelian topological group with a linearly ordered base at zero, having cofinality $\tau$.
}
\end{example}

\begin{example} \label{ex:lin}
{\em
Again, let $k=k_\tau$ from the previous example.
        Let $n$ be a natural number. The general linear group $GL_n(k)$, equipped with the topology induced from $k^n$, is a topological group with a linearly ordered base at zero, having cofinality $\tau$.
}			
\end{example}

Note that if the original ordered field was complete, the linear group over it will be complete as well. Such groups were used to construct some counter-examples in \cite{P81,P82,P98a}.

\subsection{Tukey domination}
A map $f$ between two partially ordered sets $E$ and $D$ is {\em cofinal} if the image of every cofinal subset of $E$ is cofinal in $D$. For monotone maps $f$, this is equivalent to the fact that the image of $E$ is cofinal in $D$ (\cite{DT}, Fact 5). If there is a cofinal map from $E$ to $D$, one says that $D$ is {\em Tukey reducible} to $E$, and writes $D\leq_TE$. Two partially ordered sets $E$ and $D$ are {\em Tukey equivalent,} if they are Tukey reducible to each other. Notation: $E\equiv_TD$. This is an equivalence relation, and on the equivalence classes (called {\em Tukey degrees}) $\leq_T$ is a partial order.

Among all Tukey degrees of partially ordered sets of cardinality $\leq{\mathfrak c}$, there is the top one, namely the set $[{\mathfrak c}]^{<\omega}$ of all finite subsets of the continuum with the natural order (\cite{DT}, Fact 11). To see this, recall another characterization of the Tukey order due to Schmidt \cite{Sch} (see also a proof in \cite{T}, Proposition 1). A map $f\colon D\to E$ is a {\em Tukey map}, or an {\em unbounded map,} if it sends unbounded subsets of $D$ to unbounded subsets of $E$. The existence of an unbounded map from $D$ to $E$ is equivalent to $D\leq_TE$. Now, given a partially ordered set $D$ of cardinality $\mathfrak c$, any injection $D\to [{\mathfrak c}]^{<\omega}$ is clearly unbounded.

The top Tukey order $[{\mathfrak c}]^{<\omega}$ can be used to index the bases of every uniform space of weight $\leq\mathfrak c$. 

The aim of this and subsequent short subsections is to describe 
the possible character of topological groups with
\begin{itemize} 
\item a local base at identity indexed with a {\it directed subset of  $\omega^\omega$} (Corollary \ref{l:dir_subset});  
\item a local base at identity indexed with the directed set $\omega^\omega$ (Corollary \ref{characters}).  
\end{itemize}

\begin{proposition}
        Let $(X,{\mathcal U})$ be a uniform space of weight $\leq\mathfrak c$. Then there exists a monotone cofinal map $[{\mathfrak c}]^{<\omega}\to {\mathcal U}$. In particular, if $G$ is a topological group of character $\leq\mathfrak c$, there exists a monotone cofinal map from $[{\mathfrak c}]^{<\omega}$ to the neighbourhood basis of $G$ at identity.
				\label{p:uni}
\end{proposition}

\begin{proof}
Fix a surjective map $V$ from $\mathfrak c$ onto a base of entourages of $X$, and extend it over $[{\mathfrak c}]^{<\omega}$ as a map of upper semi-lattices:
        $V(\{\kappa_1,\ldots,\kappa_n\})=\bigcap_{i=1}^n V(\kappa_i)$.
\end{proof}



\begin{proposition}
        If $G$ is a dense subgroup of the product of uncountably many non-trivial topological groups,
				then $G$ does not admit a local $\w^\w$-base.
\end{proposition}

\begin{proof} It follows from the fact
that $[\w_1]^{<\w}$ is not Tukey reducible to $\w^\w$ (a result of Isbell \cite{I2}). 

\end{proof}

In particular, now it is easy to construct
countable topological groups without a local $\w^\w$-base: just take any countable dense subgroup $G$ of the product of uncountably many discrete groups $\mathbb Z_2$.   
        An example of a different kind can be found in \cite{GabKakLei_2}, Example 6.4. 
        We will revisit countable groups later, cf. Proposition \ref{p:countable} and Example \ref{ss:ex}. 

Let us make the following perhaps known observation. 

\begin{lemma}
The directed set $\w^\w$ contains a directed subset isomorphic to $[{\mathfrak c}]^{<\omega}$.
     \label{l:dir_subset}
\end{lemma}

\begin{proof}
Let $A_{\kappa},\kappa<{\mathfrak c}$ be a family of almost disjoint subsets of $\omega$. 
        For each $\kappa$, denote the characteristic function of $A_\kappa$ by $f_\kappa$. 
        The correspondence
        \[\kappa\leftrightarrow f_{\kappa}\]
        extends to an isomorphic embedding of directed sets via
        \[\{\kappa_1,\ldots,\kappa_n\}\leftrightarrow f_{\kappa_1}\vee f_{\kappa_2}\vee\ldots \vee f_{\kappa_n}.\]
\end{proof}

As an immediate consequence of Proposition \ref{p:uni} and Lemma \ref{l:dir_subset} we obtain 

\begin{corollary}
        Every topological group of character $\leq\mathfrak c$ admits a basis of neighbourhoods of identity indexed with a directed subset of $\omega^\omega$.
				 \label{c:dir_subset}
\end{corollary}

\subsection{The character of groups with a local $\w^\w$-base}

In the following lemma, we consider a cardinal as the corresponding initial ordinal.

\begin{lemma}
        If a cardinal $\tau$ is Tukey-dominated by a partially ordered set $D$, then there is a monotone cofinal map from $D$ to $\tau$. 
        \label{l:dom}
\end{lemma}

\begin{proof} 
        Let $g\colon \tau\to D$ be a Tukey map. For every $x\in D$, define $f(x)$ as the smallest ordinal $\xi<\tau$ greater than each $\eta<\tau$ with $g(\eta)\leq x$. Clearly, $f$ is monotone and cofinal in $\tau$.
\end{proof}

We are interested in the cardinals dominated by $\w^\w$. They include $\omega$, $\mathfrak b$ (the cardinality of the smallest unbounded subset of $\w^\w$),  $cof(\mathfrak d)$ and $\mathfrak d$ (where $\mathfrak d$ is the minimal size of cofinal subset of  $\w^\w$).
For a survey and proofs with references, see \cite{mama}.
Note that implicitly the idea of calculating the actual cofinality 
of the neighborhood system at the identity in topological groups was used also in \cite{CFHT}.
 
The following is immediate with a view of Lemma \ref{l:dom}.

\begin{proposition}
        Let $\tau$ be a cardinal with $\w^\w\geq_T\tau$, and let $G$ be a topological group with a base linearly ordered by inclusion, having cofinality type $\tau$. Then $G$ has a local $\w^\w$-base.
        \label{p:dom}
\end{proposition}

\begin{corollary} 
        Every regular cardinal Tukey below $\w^\w$ is
        serving as the character of a topological group with a local $\w^\w$-base.
				\label{characters}
\end{corollary}

\begin{proof}
        Follows from Proposition \ref{p:dom} and Example \ref{ex:lin}.
\end{proof}

For $\tau={\mathfrak b}$, the last Corollary \ref{characters} answers positively Question 2.6 in \cite{GabKakLei_2}.
Moreover, since the pseudocharacter of the group from Example \ref{ex:lin} is uncountable, we get also a solution to Question 2.15 in \cite{GabKakLei_2}.
The authors are grateful to the referee who turned our attention to this observation. 

\subsection{Bankston ultraproducts}
Another source of groups with a local $\w^\w$-base are
the Bankston ultraproducts of topological groups \cite{B}. The Bankston ultraproduct of a family $\{ G_i,~i\in I \}$ of topological groups with regard to an ultrafilter $\mathcal U$ on $I$ is a quotient topological group of the box product $\square_{i\in I}G_i$ by the closed normal subgroup
\[N=\{g\in \square_{i\in I}G_i\colon \{i\in I\colon g_i=e_i\}\in{\mathcal U}\}.\]

\begin{example}
{\em
The ultraproduct of countably many copies of $\R$ is a non-standard model of the reals, ${^\ast\R}$ \cite{robinson}.
 In this case, ${^\ast\R}$ is the algebraic ultrapower of $\R$ modulo $\mathcal U$, that is, the quotient of $\R^\omega$ modulo the ultrafilter:
   \[x\sim y\iff \{i\colon x_i=y_i\}\in {\mathcal U}.\]
  As is well known, ${^\ast\R}$ forms an ordered field of an uncountable cofinal type. If equipped with the additive group structure and the order topology, ${^\ast\R}$ is a topological group, the Bankston ultraproduct of countably many additive topological groups of real numbers.
}
\end{example}

If $I$ is countable and each $G_i$ is metrizable, then clearly the Bankston ultraproduct has a base indexed with $\omega^\omega$, for instance as a quotient group of the box product of metrizable groups.


\begin{lemma}
        If the groups $G_i$ are metrizable, their Bankston ultraproduct over a non-principal ultrafilter on a countable index set even has a base of neighbourhoods at identity linearly ordered by inclusion.
        \label{l:rast}
\end{lemma}

\begin{proof}    
        The easiest way to see this is by showing that the Bankston ultraproduct is ``equipped with a metric'' taking values in ${^\ast\R}$. Fix a compatible left-invariant metric $d_n$ on each group $G_n$, whose values are bounded by one, and define a metric with values in ${^\ast\R}$ by
        \[d(x,y) = [d_n(x_n,y_n)]_{\mathcal U},\]
        the equivalence class of the corresponding element of the product. This definition does not depend on the choice of a representative of a class in $G$ (as long as we employ the same ultrafilter for both ultraproducts of course). The function $d\colon G\times G\to{^\ast\R}$ takes values in the interval $[0,1]$ of the non-standard reals, and satisfies the three axioms of a metric. Besides, this ${^\ast\R}$-valued metric is clearly left-invariant on $G$.
        Let us prove that it determines the topology. It is enough to verify this property at zero. This follows from the fact that for every $\e\in{^\ast\R}$, $\e>0$, if $\e=[(\e_n)]$, then the open ball of radius $\e$ around the identity in $G$ is nothing but
        \[B_\e(e) = \left[\prod B^n_{\e_n}(e)\right].\]
        Thus, the open balls form a basic neighbourhood system in the Bankston ultraproduct, which is obviously ordered by inclusion. In particular, if $(\e_\beta)_{\beta<\tau}$ is a transfinite sequence of strictly positive elements of ${^\ast\R}$, where $\tau$ is strictly less than the cofinality type of ${^\ast\R}$, there exists a positive $\e>0$ which forms a lower bound for all $\e_\beta$, and so every intersection of fewer than ${\mathrm{cf}}\,{^\ast\R}$ balls is a neighbourhood of identity.
\end{proof}

Now let us consider the question whether Bankston ultraproducts are Baire topological spaces. We say that a topological group $G$ is {\em spherically complete} (by analogy with a notion well known in non-archimedean analysis, cf. e.g. \cite{L,P91}), if there exists a base $\mathcal B$ of the neighbourhoods at the identity with the following property: if $\{ V_n\in {\mathcal B}:~~ n \in \omega \}$ is a (non-strictly) decreasing sequence and $\{ g_n\in G:~~n \in \omega \}$, then, provided the sequence
\[ \{ g_nV_n:~~ n \in \omega \} \]
is decreasing by inclusion, the intersection  $\bigcap_{n=1}^\infty V_n$ is non-empty.

\begin{lemma}\label{lemma1}
  Let $G$ be a spherically complete topological group.
Then $G$
  is a Baire topological space.
  \end{lemma}
  
  \begin{proof}
    Let $A$ be a meager   
    subset of $G$. There exist nowhere dense closed subsets $F_n\subseteq G$, $n \in \omega$, such that
    \[A\subseteq \bigcup_{n=1}^\infty F_n.\]
    Assume without loss in generality that the sequence $\{ F_n :n \in \omega \}$ is increasing by inclusion. Let $O$ be an arbitrary non-empty open subset of $G$. We will show that $O\not\subseteq A$.

    {\em Basis of recursion.} Since the set $F_1$ is nowhere dense in $G$, there exist $g_1\in G$ and $V_1\in{\mathcal B}$ with the property
    \[g_1V_1\subseteq O\setminus F_1.\]
    
    {\em Step of recursion.} Suppose we have chosen finite sequences $g_1,g_2,\ldots,g_n$ of elements of $G$ and $V_1\supseteq V_2\supseteq \ldots\supseteq V_n$ of nested basic neighbourhoods of $e$ in such a way that
    \[g_nV_n \subseteq g_{n-1}V_{n-1}\setminus F_{n}\subseteq \ldots \subseteq O\setminus F_1. \]
    
    Since the set $F_{n+1}$ is nowhere dense in $G$ and closed, there exist an element $g_{n+1}\in X$ and a basic neighbourhood $V_{n+1}\in {\mathcal B}$ such that 
    \[ g_{n+1}V_{n+1}\subseteq g_nV_n \setminus F_{n+1}.\] 
    Of course we may also assume $V_{n+1}\subseteq V_n$.
    We end up with a nested sequence of sets $g_nV_n$ which have the property
    \[g_nV_n\cap F_n=\emptyset.\]
    By the assumed spherical completeness 
    \[\bigcap_{n=1}^\infty g_nV_n\]
    is a non-empty 
set. It is contained in $O$ and disjoint from $A$. 
\end{proof}

Now we need some examples of spherically complete groups with linearly ordered bases of a prescribed cofinality type.
While not every non-archimedean ordered field is spherically complete, so are, for instance, saturated nonstandard models of the real numbers, and numerous other ordered fields (cf. e.g. \cite{A,hahn,LR,L,P91}). In this way, one obtains examples of Baire topological groups (linear groups over suitably chosen linearly ordered fields) with a local  $\w^\w$-base, which are not metrizable. We would prefer, however, to keep a more technical discussion of the ordered fields out of this paper, and instead concentrate on the second source of examples, the Bankston ultraproducts.

\begin{theorem}\label{Main_result_1}
        The Bankston ultraproduct, $G$, of a countable family of metrizable topological groups over a non-principal ultrafilter is spherically complete. As a consequence, 
        if all groups in the ultraproduct are non-discrete metrizable topological groups, then $G$ is
        a non-metrizable Baire group with a local $\w^\w$-base.
\end{theorem}

\begin{proof}
        Only the spherical completeness is left to prove, and it is a manifestation of the standard phenomenon of saturation in ultraproducts.
        Define a left-invariant metric $d$ on $G$ with values in ${^\ast\R}$ like in the proof of Lemma \ref{l:rast}. To witness the spherical completeness, fix a base $\mathcal B$ at identity consisting of all open balls $B_{\e}$, $\e\in{^\ast\R}$, $\e>0$. Let $g_n$ be elements of the Bankston ultraproduct $G$ and $\e_n\in{^\ast\R}$, $n \in \omega$, a sequence of positive radii, satisfying $\e_1\geq \e_2\geq \ldots\geq \e_n\geq \ldots$ and for all $n$,
        \[g_nB_{\e_n}\subseteq g_{n-1}B_{\e_{n-1}}.\]
  For each $n$ choose a representative
  \[(g_{n,1},g_{n,2},\ldots,g_{n,n},\ldots)\in \prod_{n\in\omega} G_n\]
  of $g_n\in G$ and a representative
  \[(\e_{n,1},\e_{n,2},\ldots,\e_{n,k},\ldots)\in \R^{\omega}\]
  of $\e_n$. One can assume that for every fixed $i$, the sequence $(\e_{n,i})$ is (non-strictly) monotone decreasing in $n$.
  For every $n$ select an element $U_n\in {\mathcal U}$ with
  \[d_n(g_{n,i},g_{m,i})<\e_{m,i}\mbox{ for all }m<n\mbox{ and }i\in U_n.\] We can further assume that $U_{n+1}\subseteq U_n$ for all $n$ and $\cap_nU_n=\emptyset$. Now define an element $h$ whose representative $(h_i)$ (an element of the product of $G_i$) is selected such that
  \[h_i=g_{n,i}\mbox{ whenever }i\in U_n\setminus U_{n-1}.\]
  For every $n$,
  \[\{i\in\omega\colon d_n(h_i,g_{n,i})<\e_{n,i}\}\supseteq U_n,\]
  and therefore
  \[h\in g_nV_n.\] 
\end{proof}

\section{Finest group topologies preserving a filter convergence}\label{Free_groups}

\subsection{Main result}
Here we are concerned with the following situation. Let $G$ be a group, and let $({\mathcal F}_i)$ be a family of filters on $G$. Consider the finest group topology on $G$ with regard to which each of the filters converges to the identity: ${\mathcal F}_i\to e$. When does the topological group $G$ possess a local $\w^\w$-base? Here is the main technical result.
 
\begin{theorem}
 Let $H$ be a countable topological group, and let ${\mathcal F}_n$ be a countable family of filters on $H$ each one of which admits
 an $\w^\w$-base. Suppose the topology on $H$ is the finest group topology with regard to which ${\mathcal F}_n\to e$ for every $n$.
        Then $H$ has a local $\w^\w$-base at identity.
        \label{th:gn}
\end{theorem}

The proof of Theorem \ref{th:gn} occupies the rest of the subsection. The idea is very transparent. This argument is based on a description of the finest group topology respecting a filter convergence developed by Roelcke and Dierolf \cite{RD}. The basic neighbourhoods in this topology are indexed with sequences of functions from the group $H$ to a filter base, and this correspondence is monotone. Countably many filters can be equivalently replaced with a single filter having an $\w^\w$-base. It follows that there is a neighbourhood basis at $e$ indexed with $\left((\omega^\omega)^H\right)^\omega\cong \omega^\omega$. 

Let us start with the following observation.
It should be obvious that each of the filters ${\mathcal F}_n$ converges to identity if and only if the intersection filter ${\mathcal F}=\bigcap_{n\in\omega} {\mathcal F}_n$ converges to the identity.
Thus, under Theorem's hypotheses, the topology of $H$ is the finest group topology under which 
 ${\mathcal F}$ converges to identity. 
Hence the following standard argument reduces the proof to the case where there is only a single filter to consider.

\begin{lemma}
        Let $\{{\mathcal F}_n:~~ n\in\omega \}$ be a countable family of filters on a set $X$. Suppose that each filter ${\mathcal F}_n$ admits an $\w^\w$-base. Then the filter ${\mathcal F}=\bigcap_{n\in \omega} {\mathcal F}_n$ admits an $\w^\w$-base.
\end{lemma}

\begin{proof} 
        For every $n\in \omega$ fix a monotone cofinal map $V_n\colon \omega^\omega\to {\mathcal F}_n$. For each sequence of elements
         $f=(f_n)_{n \in \omega}$, where $f_n\in \omega^\omega$ (that is, $f$ is an element of
         $\left(\omega^\omega \right)^{\omega}\cong \omega^{\omega\times \omega}$) define the set
        \[V(f) = \bigcup_{n\in \omega} V_{n}(f_n)\subseteq X.\]
        Clearly, the sets of this kind form a base for the intersection filter $\cap_n {\mathcal F}_n$. Moreover, if $f,g\in \omega^{\omega\times \omega}$ and $f\leq g$ pointwise, the assumption on the monotonicity of $V_n$ implies that $V(f)\subseteq V(g)$. 
\end{proof}

We find it useful to present a summary of the technique developed by Roelcke and Dierolf, as well as for the reader's convenience to sketch their proofs, because the book \cite{RD} is out of print.
Denote by $S_n$ the symmetric group of all bijections of the set $n=\{0,..., n-1\}$.
For a sequence $\{B_n\}_{n=1}^{\infty}$ of subsets of a group $G$, define their symmetric product as follows:
\[\sym{ B_n}=\sym{ B_n}_{n=1}^{\infty} = \bigcup_{n=1}^{\infty}\bigcup_{\sigma\in S_n} B_{\sigma(1)}\cdot B_{\sigma(2)}\cdot \ldots \cdot B_{\sigma(n)}.\]

  Let now $\mathcal F$ be a filter of subsets of a group $G$.
  For a mapping $\Phi\colon G\to {\mathcal F}$ denote
  \begin{equation}
          {\mathcal V}_{\Phi} =\bigcup_{g\in G} g^{-1}(\Phi(g)\cup\Phi(g)^{-1})g.
          \label{eq:union}
  \end{equation}
  According to \cite{RD}, the sets of the form 
  \[\sym{{\mathcal V}_{\Phi_n}}_{n=1}^{\infty},\] 
  where $(\Phi_n)$ runs over all sequences of maps from $G$ to $\mathcal F$, 
  form a neighbourhood base at identity in the finest group topology on $G$ in which ${\mathcal F}\to e$. 
  This is a consequence of the following simple results.
  
    \begin{lemma}
            Every set $\sym{{\mathcal V}_{\Phi_n}}$ is symmetric. \qed
  \end{lemma}
  
  \begin{lemma} Assuming the sequence of mappings $\Phi_n\colon G\to {\mathcal F}$ is pointwise monotone (that is, $\Phi_{n+1}(g)\subseteq \Phi_n(g)$ for each $g\in G$), we have 
    \[\sym{{\mathcal V}_{\Phi_{2n}}}^2\subseteq\sym{{\mathcal V}_{\Phi_n}}.\]
    \qed
  \end{lemma}
 
  \begin{lemma} For each $h\in G$, denote $\Phi^h$ the right translate of $\Phi$ by $h$, that is, $\Phi^h(g)=\Phi(gh)$. Then
    \[h^{-1}\sym{{\mathcal V}_{\Phi^h_n}}h\subseteq \sym{{\mathcal V}_{\Phi_n}}.\]
    \qed
  \end{lemma}
  
  Also, it is clear that every set of the form $\sym{{\mathcal V}_{\Phi_n}}$ contains an element of the filter $\mathcal F$, so ${\mathcal F}\to e$ with regard to this group topology. It remains to prove that the resulting group topology is indeed the finest one making the filter converge. 
  First of all, the proof of the Birkhoff--Kakutani lemma about continuous pseudometrics on topological groups implies:
  
\begin{lemma}
  If $\{V_n\}_{n=1}^{\infty}$ be a sequence of neighbourhoods of the identity in a topological group $G$ such that $V_n^{-1}=V_n$ and $V_{n+1}^2\subseteq V_n$ for all $n$. Then for every $k$ one has
\[
\sym{V_n}_{n=k+2}^{\infty}\subseteq V_k.\]
\label{l:bk}
\qed
\end{lemma}

Now let $G$ be equipped with some group topology making $\mathcal F$ converge to $e$. Fix a neighbourhood $V$ of identity in this group topology. We will embed into $V$ a set of the form $\sym{{\mathcal V}_{\Phi_n}}_{n=1}^{\infty}$ for a suitably chosen sequence $(\Phi_n)$. 

Find a sequence $\{V_n\}_{n=1}^{\infty}$ of symmetric neighbourhoods of the identity with $V_{n+1}^2\subseteq V_n$ for all $n$ and $V_0=V$. For each $n$, select an $F_n\in{\mathcal F}$ with $F_n\subseteq V_n$. Given a $g\in G$ and $n\in \omega$, there is $m\geq n$ with $g^{-1}V_mg\subseteq V_n$. In particular, $g^{-1}F_mg\subseteq V_n$. Set $\Phi_n(g)=F_m$. Now it follows that for every $n$,
\[{\mathcal V}_{\Phi_n}=\bigcup_{g\in G} g^{-1}(\Phi_n(g)\cup\Phi_n(g)^{-1})g\subseteq V_n,\]
and by Lemma \ref{l:bk},
\[\sym{{\mathcal V}_{\Phi_n}}_{n=1}^{\infty}\subseteq V.\]

We have just reproved the result by Roelcke and Dierolf. 

To conclude the proof of Theorem \ref{th:gn}, fix a base, $\mathcal B$, for the filter $\mathcal F$, indexed with $\omega^\omega$. According to the above result by Roelcke and Dierolf, there is a base at identity for the finest group topology on $H$ with $\mathcal F\to e$, indexed with sequences of maps $H\to {\mathcal F}$, that is, elements of the partially ordered set
  \[\left({\mathcal B}^{H}\right)^{\omega} \cong\left(\omega^\omega\right)^{H\times\omega}\cong \omega^{\omega\times H\times \omega}.\]
It follows from the definition of the basic neighbourhoods $\sym{{\mathcal V}_{\Phi_n}}_{n=1}^{\infty}$
that, if for every $h\in H$ and $n\in\N$ one has $\Phi_n(g)\subseteq\Psi_n(g)$, then
\[\sym{{\mathcal V}_{\Phi_n}}_{n=1}^{\infty}\subseteq \sym{{\mathcal V}_{\Psi_n}}_{n=1}^{\infty}.\]
In other words, the map 
\[ \left(\omega^\omega\right)^{H\times\omega}\ni (\Phi_n)_{n=1}^{\infty} \mapsto \sym{{\mathcal V}_{\Phi_n}}_{n=1}^{\infty}\]
is monotone. We are done.
  
  In \cite{P} the above technique was used to describe a neighbourhood base in the free topological group on a uniform space, and in \cite{PU}, to construct an example of a projectively universal countable metrizable group.
 
We do not know if ``countable'' in the statement of Theorem \ref{th:gn} can be replaced with ``separable'', but the result leads to a number of new and important corollaries for separable topological groups.

 \subsection{Corollaries for free products}

 The {\em free product} of a family of topological groups $G_i$, $i\in I$ is a coproduct in the category of topological groups. 
This construction was investigated by Graev \cite{Gr}, who has shown that the free product $\ast_{i\in I}G_i$ is algebraically isomorphic to the algebraic free product, and that every topological group $G_i$ canonically embeds into the free product as a topological subgroup (this is not obvious). The free product is defined by the universality property: every collection of continuous homomorphisms $h_i\colon G_i\to G$ to some common topological group $G$ uniquely extends to a continuous homomorphism $\ast_{i\in I}G_i\to G$. Equivalently, $\ast_{i\in I}G_i$ is the algebraic free product of the groups $G_i$, given the finest group topology inducing the original topology on each subgroup $G_i$. 

\begin{corollary}
        The free product $G={\ast}_{n\in \omega}G_n$ of countably many separable topological groups each having a local $\w^\w$-base has a local $\w^\w$-base. In particular, this conclusion holds for the free product of countably many Polish groups. \qed
\end{corollary}

\begin{proof}
Select a countable dense subgroup $H_n$ in each $G_n$, and let $H$ be algebraically generated by $H_n,n\in\N$ in $G$. Algebraically, $H$ is the free product of $H_n$. Put on $H$ the finest group topology inducing the original topologies on $H_n$. Such a topology exists and is at least as fine as the induced topology from $G$. The canonical embedding of $H$ (with the new topology) into $G$ is continuous and so extends over two-sided completions of the two groups, $i\colon\hat H\to\hat G$. The restriction of $i$ to each $\hat H_n$ is a topological group isomorphism with $\hat G_n$, and in this way
every topological group $G_n$ is canonically identified with a subgroup of $\hat H$;
denote $\tilde H$ the topological subgroup of $\hat H$ generated by all such copies of groups $G_n$. Since $G$ is an algebraic free product, it follows from the universality property that 
the restriction of $i$ to $\tilde H$ is a (continuous) algebraic isomorphism with $G$. The topology of $\tilde H$ cannot be finer than the topology of the free product $G$, and so $i\vert_{\tilde H}$ is an isomorphism of topological groups.  

The topology of $H$ as defined above is the finest topology with regard to which each neighbourhood filter of $H_n$ converges to $e$. By Theorem \ref{th:gn}, $H$ has an $\w^\w$-base. The same holds for $\hat H$, $\tilde H$, and therefore $G$.
\end{proof}

Here is a formally more general, but in fact equivalent, statement.

\begin{corollary}
        Let $G$ be a topological group which is topologically generated by countably many separable subgroups $G_n$ each having a local $\w^\w$-base.
         Suppose the topology on $G$ is the finest group topology inducing the given topologies on each $G_n$. 
         Then $G$ admits a local $\w^\w$-base.  
         \label{th:chain}
 \end{corollary}

\begin{proof}
Denote $\tilde G$ a subgroup of $G$ algebraically generated by the $G_n$.
The canonical continuous surjective homomorphism $\ast_nG_n\to G$ is in fact an open map, because the quotient topology induces the original topology on each $G_n$ and hence must be coarser than the topology of $G$. And a local $\omega^{\omega}$ base is preserved by quotient homomorphisms. Finally, the existence of an $\omega^{\omega}$ base passes over to $G$.
\end{proof}

Already for a group like $U(\ell^2)\ast U(\ell^2)$ (where the unitary group is equipped with the standard strong operator topology) the above conclusion is new.

\subsection{Corollaries for free topological groups}
For a Tychonoff space $X$, the free topological group on $X$ contains $X$ as a topological subspace, is algebraically generated by $X$, and every continuous map from $X$ to a topological group lifts to a continuous homomorphism of $F(X)$, see \cite{Gra,Mar}. Similarly, one defines the free Abelian topological group $A(X)$.

It is sometimes better to work with more general notions defined for a uniform space $X$ instead of a topological one. The definition of the free topological group on a uniform space $X$ parallels that for a topological space, except that there are at least four standard uniformities on a topological group, so we need to pick up one; for Nummela's definition \cite{nummela}, one chooses the two-sided uniformity, which is the supremum of the left and right uniformities. Thus, $X$ is a uniform subspace of $F(X)$ in its two-sided uniformity, and every map from $X$ to a topological group $G$, which is uniformly continuous with regard to the two-sided uniformity on $G$, lifts to a continuous group homomorphism. 

\begin{remark}
{\em
It appears to us that the other two possible notions, using the left uniformity and the lower uniformity, have never been explored at any depth, although they may lead to interesting new examples; for some relevant work, see \cite{tkachenko}.
}
\end{remark}

\begin{corollary}
        Let $X$ be a separable uniform space admitting an $\w^\w$-base of entourages of the diagonal. Then the free topological group $F(X)$ admits a local $\w^\w$-base at identity. 
\end{corollary}

\begin{proof}
According to a result of Nummela \cite[Theorem 4]{nummela}, if $Y$ is a dense uniform subspace of a uniform subspace $X$, then the canonical continuous monomorphism of $F(Y)$ into $F(X)$ is a (dense) embedding of topological groups.
Consequently, it suffices to show that the topological group $F(Y)$ has a local $\w^\w$-base, because $F(X)$ contains an isomorphic dense copy of $F(Y)$.
 
        Denote $F(Y)$ by $H$. Define a convergent filter $\mathcal F$ on $H$ as follows.
				For every entourage $V$ from the uniform structure $\mathcal V$ on $Y$ associate the set
				\[i(V) = \{x^{-1}y\colon (x,y)\in V\}\cup \{xy^{-1}\colon (x,y)\in V\}.\]
				
It is now easy to see that the canonical embedding of $Y$ into $F(Y)$ equipped with some group topology is uniformly continuous in the corresponding two-sided uniformity if and only if the filter $(i(V))_{V\in {\mathcal V}}$ converges to identity. This was explored in \cite{P}.
Denote 	$\mathcal F = \{ i(V): V \in \mathcal V \}$. Clearly, this filter $\mathcal F$ admits an $\w^\w$-base if and only if so does uniform structure $\mathcal V$.
Since $H = F(Y)$ is countable, we conclude by Theorem \ref{th:gn}.
\end{proof}

It is easy to see that the free topological group on the Tychonoff space $X$ is the same as the free topological group on the uniform space $X$ equipped with the finest compatible uniformity \cite{nummela}. One therefore deduces:

\begin{corollary}
        Let $X$ be a separable topological space whose finest compatible uniformity admits an $\w^\w$-base. Then the free topological group $F(X)$ on $X$ admits a local $\w^\w$-base. \qed
\end{corollary}

The topological spaces whose finest uniformity consists of all neighbourhoods of the diagonal are known as strongly collectionwise normal spaces \cite{SS}. For instance, all paracompact spaces are such, but non-normal spaces are not. However, one can complement the above result as follows.

\begin{corollary}
   Let $X$ be a separable topological space whose neighbourhood system of the diagonal $\Delta_X$ in $X\times X$ admits
	an $\w^\w$-base.      
   Then the free topological group $F(X)$ on $X$ admits a local $\w^\w$-base.
\label{c:nbhd}
\end{corollary}

In order to prove Corollary 3.12 we need a result of uniform topology.

\begin{lemma}
Let $Y$ be a dense subspace of a topological space $X$, and let $\mathcal V$ be a compatible uniformity on $Y$. Suppose that for every $V\in\mathcal V$ the closure ${\mathrm{cl}}_{X\times X}(V)$ is a neighbourhood of the diagonal in $X$. Then the uniformity $\mathcal V$ is a restriction of some compatible uniformity from $X$. (Here we understand compatibility in a weak sense: every set of the form $V[x]$ is a neighbourhood of $x$.)
\label{l:denses}
\end{lemma}

\begin{proof}
The sets ${\mathrm{cl}}_{X\times X}(V)$, $V\in {\mathcal V}$ contain the diagonal of $X$ and form a filter base. The family of such sets is closed under the flip $(x,y)\mapsto (y,x)$.
To prove that they form a base for a uniformity on $X$, it remains to show that for every $V\in {\mathcal V}$ there is $W$ with ${\mathrm{cl}}_{X\times X}(W)^2\subseteq {\mathrm{cl}}_{X\times X}(V)$. 

Choose a symmetric $W$ with $W^4\subseteq V$, and let $x,y,z\in X$ and $(x,y),(y,z)\in {\mathrm{cl}}_{X\times X}(W)$. Select converging nets of elements of $Y$, $x_\alpha\to x$, $y_\alpha,y^\prime_\alpha\to y$, $z_\alpha\to z$ so that for all $\alpha$, $(x_\alpha, y_\alpha)\in W$, $(y^\prime_\alpha,z_\alpha)\in W$. Eventually one has $(y_\alpha,y^\prime_\alpha)\in {\mathrm{cl}}_{X\times X}(W)$ (as by hypothesis the former set is a neighbourhood of $(y,y)$ in $X\times X$). Therefore, eventually, $(y_\alpha,y^\prime_\alpha)\in {\mathrm{cl}}_{Y\times Y}(W)\subseteq W^2$. We conclude: $(x_\alpha,z_\alpha)\in W^4$ and $(x,z)\in {\mathrm{cl}}\,(W^4)\subseteq {\mathrm{cl}}\,(V)$. 

The uniformity on $X$ so defined is compatible because for every $x\in X$ the set ${\mathrm{cl}}_{X\times X}(V)[x]$ is a neighbourhood of $x$. The restriction of this uniformity to $Y$ is $\mathcal V$.
%
\end{proof}

\begin{proof}[Proof of Corollary \ref{c:nbhd}]
Select a countable dense subset $Y$ of $X$, and denote $F(Y)$ the free subgroup generated by $Y$. Let $\mathcal N$ denote the restriction of the neighbourhood system of the diagonal of $X$ to $Y\times Y$. Thus, $\mathcal N$ is coarser than the filter of the neighbourhoods of the diagonal in $Y$ (since $Y$ is countable, it is the universal uniform structure on $Y$), yet finer than the restriction of the universal uniformity, ${\mathcal U}_X$, of $X$ to $Y$. Notice that the closure in $X\times X$ of every element of $\mathcal N$ is a neighbourhood of the diagonal in $X\times X$.

        Consider the image, $\mathcal F$, of $\mathcal N$ under the map $i$ sending each $V$ to the collection of all words $x^{-1}y$ and $xy^{-1}$, whenever $(x,y)\in V$. 

        Let us put on $F(Y)$ the finest group topology making $\mathcal F$ converge to the identity. This topology is apriori finer than the topology of the free topological group on the uniform space $(Y,{\mathcal U}_X\vert_Y)$, which is, according to Nummela's theorem, the restriction of the topology from the free topological group $F(X)$. 

Denote $\mathcal V$ the restriction to $Y$ of two-sided uniformity on $F(Y)$ equipped with the above group topology. Since $\mathcal N\to e$, every element $V\in\mathcal V$ contains an element of $\mathcal N$, and so ${\mathrm{cl}}_{X\times X}(V)$ is a neighbourhood of the diagonal in $X$. 
Using Lemma \ref{l:denses}, we conclude that $\mathcal V$ is the restriction of the finest uniformity from $X$. Thus, $F(Y)$ is the free topological group $F(Y,{\mathcal U}_X\vert_Y)$, which is a dense topological subgroup of $F(X)$. Since $F(Y)$ has a local $\w^\w$-base, so does $F(X)$.
\end{proof}

According to \cite{Ginsburg}, a Tychonoff space $X$ has a diagonal of countable character if and only if 
$X$ is metrizable and the set $X^{\prime}$ of non-isolated points of $X$ is compact. 
 
\begin{proposition}\label{sigma-compact}
Let X be a metrizable space such that the set $X^{\prime}$ of all non-isolated points in $X$ is $\sigma$-compact, i.e.
$X^{\prime}$ is a countable union of compacts. Then the finest uniformity of the space $X$ has an $\w^\w$-base of entourages.
\end{proposition}

\begin{proof}
        Since $X$ is paracompact, it is enough to prove that the neighbourhood system of the diagonal in $X\times X$ admits a local $\w^\w$-base of neighbourhoods. Fix a metric $d$ generating the topology of the metrizable space $X\times X$. 
        Let $X^{\prime}=\bigcup_{n\in\omega}K_n$, where $K_n$ are compact. Denote by $\tilde K_n$ the compact image of $K_n$ under the diagonal map  
        $\mathrm{Id}_{X}{\Delta}\mathrm{Id}_{X}$.  For each $\alpha\in \omega^\omega$, define
        \[U_{\alpha} = \bigcup_{n\in\omega} (\tilde K_n)_{2^{-\alpha(n)}},\]
        where $A_{\e}$ is the open $\e$-neighbourhood of a set $A$ under the metric $d$ formed in $X\times X$.
Define the map $f\colon \alpha \mapsto V_{\alpha}$ from $\omega^\omega$ to the family of open neighbourhoods of the diagonal 
$\Delta_X$ in the square $X \times X$ as follows:
   \[f(\alpha) = U_\alpha  \bigcup \Delta_X .\]

The map $f$ is obviously monotone. To show that $f$ is cofinal, just remember an easily verifiable and well-known fact that the $\e$-neighbourhoods of a compact subset of a metric space form a neighbourhood base of this compact.
Thus, the family
$\mathcal{V} =\{ V_\alpha = U_\alpha  \bigcup \Delta_X \}_{\alpha\in \omega^\omega}$ is an $\w^\w$-base of the diagonal $\Delta_X$ in the square $X \times X$.       
\end{proof}

\begin{conjecture}\label{Conj_1}
 Conversely, if the finest uniformity of a metrizable space $X$ has an $\w^\w$-base of entourages, then the set $X^{\prime}$ of all non-isolated points in $X$ is $\sigma$-compact.
\footnote{Very recently Taras Banakh resolved positively Conjecture 3.15.}
\end{conjecture}

  \subsection{The case where left and right uniformities coincide}
  In the case where $G$ is a SIN group (that is, the left and the right uniformities of $G$ coincide), the separability condition can be dropped. In this case, the following more or less obvious observation holds.
  
  \begin{lemma}
          Let $G$ be a group, and $\mathcal F$ be a filter of subsets of $G$. A base at identity for the finest SIN group topology in which ${\mathcal F}\to e$ is formed by all sets of the form
          \[\sym{\left(\bigcup_{g\in G} g^{-1} V_n g\right)_{n}},\]
          where $\{V_n\in {\mathcal F}:~~ n\in \omega\}$ is an arbitrary sequence of elements of the filter. \qed
  \end{lemma}
  
  In other words, we only take those $\Phi_n$ which take constant values on all of $G$. This follows directly from the Roelcke--Dierolf description and the conjugation-invariance of the basic neighbourhoods of the group (which is a reformulation of the SIN property, hence the name: small invariant neighbourhoods).
  
  In this way, there is a base at identity which is indexed with all countable sequences of elements of a base of the filter. This leads immediately to:

\begin{theorem}
        Let $G$ be a SIN group, and let ${\mathcal F}_n$ be a countable family of filters on $G$ each one of which admits an $\w^\w$-base. Suppose  the topology on $G$ is the finest SIN topology with regard to which ${\mathcal F}_n\to e$ for every $n$.  
        Then $G$ has a local $\w^\w$-base at identity. \qed
\end{theorem}

Recall that a uniform space $X$ has an $\w^\w$-base if the filter of entourages of the diagonal $\Delta_X$ in the square $X \times X$ 
has an $\w^\w$-base.

\begin{corollary}
        The free SIN group on a uniform space $X$ has a local $\w^\w$-base if and only if the diagonal $\Delta_X$ in the square $X \times X$ 
admits an $\w^\w$-base of entourages.
\end{corollary}

In particular, this holds if $G$ is Abelian. 

\begin{corollary}\label{cor_abelian}
        The free Abelian topological group $A(X)$ on a uniform space $X$ admits a local $\w^\w$-base if and only if the diagonal $\Delta_X$ in the square $X \times X$ 
has an $\w^\w$-base of entourages.
\end{corollary}

\begin{corollary}\label{cor_abelian_tyh}
        The free Abelian topological group $A(X)$ on a Tychonoff space $X$ admits a local $\w^\w$-base if and only if the finest compatible uniformity on $X$ 
				has an $\w^\w$-base.
\end{corollary}

As an immediate consequence of the last Corollary \ref{cor_abelian_tyh} and Proposition  \ref{sigma-compact}  we have the following

\begin{corollary}\label{cor_sigma-compact}
        Let X be a metrizable space such that the set $X^{\prime}$ of all non-isolated points in $X$ is $\sigma$-compact.
        Then the free Abelian group $A(X)$ has a local $\w^\w$-base. \qed
\end{corollary}
 
Note that in view of Example \ref{ss:ex} the metrizability assumption on $X$ can not be dropped.
\begin{remark}
{\em
Corollary \ref{th:chain} from the previous section and Corollary \ref{cor_sigma-compact} generalize substantially several results from \cite{GabKakLei_2} and \cite{GabKak_2}.
} 
\end{remark}

Finally, we obtain a characterization of countable spaces whose free topological group admits a local $\w^\w$-base.

\begin{proposition} 
        \label{p:countable}
        A countable space $X$ has an $\w^\w$-base of the diagonal if and only if the filter of neighbourhoods of each point of $X$ admits an $\w^\w$-base. 
\end{proposition}

\begin{proof} $\Rightarrow$ is trivially true, and to prove $\Leftarrow$, given a monotone cofinal map $i_x$ for the neighbourhood filter of every point, coalesce them all together into a monotone map $i$ from $\left(\omega^\omega\right)^X$: for $f\colon X\to \omega^\omega$, set
        \[i(f)=\bigcup_{x\in X} i_x(f(x))\times i_x(f(x)).\]
\end{proof}

\begin{example}\label{ss:ex}
{\em
 We show that even for a countable space $X$ with one non-isolated point, the topological groups $A(X)$ and $F(X)$ need not have a local $\w^\w$-base.
 
We view ultrafilters as partially ordered sets in a usual way, that is, for $A,B\in\xi$ we take $A\leq B\iff A\supseteq B$. 
Let us mention first that there are ultrafilters which are Tukey strictly above $\omega^\omega$.
 Indeed, according to Isbell \cite{I}, in ZFC there exist ultrafilters $\xi$ of the top Tukey type, that is, such that $\xi\equiv_T [{\mathfrak c}]^{<\omega}$; as noted in \cite{DT}, in a comment after Isbell's theorem 17, such ultrafilters are very numerous, as in fact there are $2^{\mathfrak c}$ of them. (It is unknown whether ultrafilters strictly below the Tukey top exist in ZFC.)
 
The authors are grateful to Stevo Todorcevic for pointing out to us
 that in fact, as follows from results of Solecki and Todorcevic \cite{SolTod}, no non-principal ultrafilter is Tukey-reducible to $\omega^\omega$.
As we were unable to find this deduction in an explicit form anywhere in the literature, we include it for reader's convenience.
 
 Let $D$ be a filter of subsets of $\omega$ and assume that $D \leq_{T} \omega^\omega$.
 Additionally, we view $D$ as a subset of the Cantor set $\{0,1\}^{\omega}$.
 Thus $D$ is a metric separable space with a partial order in which the
set of predecessors of each element is compact. 
As a topological space the set $\omega^\omega$ is of course analytic, and we deduce, by \cite{SolTod}, Corollary 5.4, 
that $D$ is also analytic.
However, it is well known that no non-principal ultrafilter is analytic.
 
 Now let  $\xi$ be any non-principal ultrafilter. We form the countable topological space $X=\omega\cup\{\xi\}$ considered as a subspace of $\beta\omega$ with the induced topology. The filter of neighbourhoods of $\xi$ in $X$ being Tukey equivalent to $\xi$, 
is not Tukey-reducible to $\omega^\omega$, therefore the topological groups $A(X)$ and $F(X)$ do not have a local $\w^\w$-base.
}
\end{example}

\subsection{Free locally convex space}
We conclude the paper by discussing the free locally convex space $L(X)$ on a topological or uniform space $X$ \cite{Flo2,Rai,Usp} and answering another question from \cite{GabKakLei_2}. This space has $X$ as a Hamel basis, and enjoys the same universal property as the free topological group with regard to all (uniformly) continuous maps from $X$ to locally convex spaces: every such map extends to a continuous linear map from $L(X)$. 

As a consequence of a standard result in the theory of mass transportation, the free Abelian topological group $A(X)$ canonically embeds into $L(X)$ as an additive topological subgroup generated by $X$ (as shown in \cite{Tkach},\cite{Usp2}).

It was asked in \cite{GabKakLei_2} (Question 4.19) whether the space $L(X)$ admits a local $\w^\w$-base if and only if the group $A(X)$ does. Necessity follows from our remark above. We will show that the sufficiency does not hold.

For $X$ discrete, the free locally convex space $L(X)$ is just the direct sum of $\tau=\abs X$ many copies of $\R$ equipped with the finest locally convex topology. We will denote such a sum $\R^{\oplus\tau}$. In the countable case, it is well known and easily verified that the finest locally convex topology on $\R^{\oplus\omega}$ equals each of the following two topologies: the topology of a chain with regard to a cover by finite-dimensional subspaces $\R^n$, $n\in\omega$, and also the box topology induced from $\R^{\square \omega}$. For uncountable $\tau$, neither of the two assertions is true. However, the box product topology, being locally convex, is contained in the finest locally convex topology. This leads to:

\begin{lemma}
        The neighbourhood filter, $\mathcal N$, of the space $\R^{\oplus\tau}$ Tukey dominates $\omega^{\tau}$ with the pointwise order.
\end{lemma}

\begin{proof}
        For each $f\in \omega^{\tau}$ define a box product neighbourhood:
        \[\square_f=\prod_{\beta<\tau} B_{1/f(\beta)}(0)\cap \R^{\oplus\tau}.\]
        We claim that the map above, viewed as a monotone map to the neighbourhood system $\mathcal N$, is unbounded. Indeed, unboundedness of a set of functions $A\subseteq \omega^{\tau}$ means exactly that for some $\beta<\tau$,
        \[\sup_{f\in A}f(\beta)=\infty.\]
        Consequently,
        \[x\in\bigcap_{f\in A}\square_f \Rightarrow x_{\beta}=0.\]
        Thus, $\bigcap_{f\in A}\square_f$ is not a neighbourhood of identity in the finest locally convex topology either. 
\end{proof}

The following is proved by the diagonal argument.

\begin{lemma}
        The directed set $\omega^{\tau}$ has cofinality type $>\tau$ and hence is Tukey above $\w^\w$ for every $\tau\geq{\mathfrak c}$.
\end{lemma}

\begin{proof}
        Let $A$ be a subset of $\omega^{\tau}$ of cardinality $\tau$. Enumerate $A=\{a_{\beta}\colon\beta<\tau\}$. Define $z \in \omega^{\tau}$:
        \[z(x) = a_x(x) +1.\]
Then for any $a \in A$ is not true that $z < a$.
\end{proof}

\begin{example}
{\em
Let $X$ be a discrete topological space of cardinality $\geq\mathfrak c$. Then the group $A(X)$ is clearly discrete, and thus has a local $\w^\w$-base. At the same time, it follows from the above results that the smallest cardinality of a base at identity for $L(X)$ is at least ${\mathfrak c}^+$, and so $L(X)$ does not have a local $\w^\w$-base.
}
\end{example}
It would be interesting to characterize those topological (or uniform) spaces $X$ for which the free locally convex space $L(X)$ has a local $\w^\w$-base.

\subsection*{Acknowledgments}
We thank the anonymous referee who found a mistake in the proof of the original version of Theorem \ref{th:gn} and made a number of very useful remarks.

The authors are grateful to Natasha Dobrinen, Claude Laflamme, and Stevo Todorcevic for helpful remarks.
Thanks to the remarks of Saak Gabri-
yelyan, Jerzy~K{\c{a}}kol and Taras Banakh several minor inaccuracies were corrected.

This investigation started when the authors met together at the conference ``Functional Analysis and Group Actions'' in February 2014 in Florian\'opolis, Brazil, supported from the Special Visiting Researcher project of the program Science Without Borders of CAPES, Brazil (processo 085/2012). The first-mentioned author gratefully  acknowledges the  financial support he received from the University of S\~ao Paulo and the above-mentioned PVE project. The second-mentioned author (V.G.P.) was the Special Visiting Researcher of this project.

\end{document}